\documentclass[reqno,a4paper,12pt]{amsart}
\usepackage{amsmath,amsthm,amsfonts,amssymb, mathrsfs}
\usepackage{verbatim, graphicx, ifthen, enumitem}
\usepackage[T1]{fontenc}
\usepackage [applemac] {inputenc}
\usepackage[svgnames]{xcolor}

\allowdisplaybreaks

\let\oldmarginpar\marginpar
\renewcommand\marginpar[1]{\-\oldmarginpar[\raggedleft\footnotesize #1]%
{\raggedright\footnotesize #1}}
\newtheorem{theoa}{Theorem}

\newtheorem{theorem}{Theorem}
\newtheorem*{theorema'}{Theorem A'}
\newtheorem*{theorem1'}{Theorem 1'}
\newtheorem*{theorem2'}{Theorem 2'}
\newtheorem*{theorem3'}{Theorem 3'}
\newtheorem*{theorem4'}{Theorem 4'}
\newtheorem{lemma}{Lemma}

\newtheorem*{proposition}{Proposition}

\theoremstyle{definition}
\newtheorem{definition}{Definition}

\theoremstyle{remark}

\newcommand{\Z}{\mathbb{Z}}

\newcommand{\R}{\mathbb{R}}

\newcommand{\C}{\mathbb{C}}

\def\L{\Lambda}
\def\l{\lambda}

\def\Z{\mathbb{Z}}

\def\R{\mathbb{R}}
\def\C{\mathbb{C}}

\def\1{\mathbf{1}}

\author{ANDREW AHN}
\address{Department of Mathematics, Massachusetts Institute of Technology}
\email{ajahn@mit.edu}

\author{WILLIAM CLARK}
\address{Department of Mathematics, University of Michigan}
\email{wiclark@umich.edu}

\author{Shahaf Nitzan}
\address{School of Mathematical Sciences, Georgia Institute of Technology}
\email{shahaf.nitzan@math.gatech.edu }

\author{JOSEPH SULLIVAN}
\address{Department of Physics, Yale University}
\email{joseph.m.sullivan@yale.edu}


\begin{document}

 \title{Density of Gabor systems via the short time Fourier transform}

 \begin{abstract}
 We apply a new approach to the study of the density of Gabor systems, and obtain a simple and straightforward proof to Ramanathan and Steger's well--known result regarding the density of Gabor frames and Gabor Riesz sequences. Moreover, this point of view allows us to extend this result in several directions. The approach we use was first observed by A.~Olevskii and the third author in their study of exponential systems, here we develop and simplify it further.
 \end{abstract}
 \maketitle
\section{introduction}

For $g \in L^2(\R)$ and $\L\subset\R^2$, the Gabor system generated by $g$ and $\L$ is
denoted by
\begin{equation}\label{gabor}
G(g, \L) := \{e^{2\pi i\mu t}g(t - \l)\}_{(\l,\mu)\in \L}.
\end{equation}
Such systems play a prominent role in time-frequency analysis and its applications. For example, they appear in the study of pseudo-differential operators and phase space concentration and are used in speech processing, the analysis of musical signals, wireless transmission and quantum mechanics. The study of these systems focuses on the interplay between the properties of $G(g,\L)$, the distribution of $\L$, and the
time-frequency localization of $g$.

In this paper we focus on questions of the following form: which conditions on the distribution of $\L$ are necessary for the system $G(g,\L)$ to have certain properties? In particular, we are interested in the conditions on $\L$ implied by the system $G(g,\L)$ being a Riesz basis, a frame (roughly speaking, an `over--complete' basis), a Riesz sequence
(an `under--complete' basis) and a uniformly minimal system (a very relaxed version of a Riesz sequence).

Such questions were first considered in this setting by Ramanathan and Steger who, in particular, obtained the following well known result \cite{RS} (for the definition of a uniformly discrete sequence see subsection 2.2).
\begin{theoa}$[$Ramanathan and Steger$]$
Let $g\in L^2(\R)$ and let $\L\subset\R^2$ be a uniformly discrete sequence.
\begin{itemize}
\item[1.] If $G(g,\L)$ is a Riesz sequence then
\begin{equation}\label{RS-RS}
D^+(\L):=\lim_{R\rightarrow\infty}\frac{\max_{(a,b)\in\R^2}|\L\cap Q_{(a,b)}(R)|}{(2R)^2}\leq 1,
\end{equation}
\item[2.] If $G(g,\L)$ is a frame then
\begin{equation}\label{RS-frame}
D^-(\L):=\lim_{R\rightarrow\infty}\frac{\min_{(a,b)\in\R^2}|\L\cap Q_{(a,b)}(R)|}{(2R)^2}\geq 1,
\end{equation}
\end{itemize}
where $Q_{(a,b)}(R)$ denotes the cube with center at $(a,b)$ and side length $2R$.
\end{theoa}

To obtain this result, Ramanathan and Steger applied a technique which involved trace estimates of compositions of certain projection operators. The density of Gabor systems has since become a significant area of research, see \cite{C1} and references therein for an extensive review of the area.

Prior to this Landau studied a similar question regarding sampling and interpolating sequences for functions with a given bounded spectrum \cite{La}. Landau's technique required a much more delicate estimate of the eigenvalues of compositions of projection operators. This technique, however, could be applied in a more general setting. Landau's technique, as well as Ramanathan and Steger's, has been applied in different settings, e.g in the study of sampling and interpolating sequences in analytic function spaces and in that of Fekete points on complex manifolds, (\cite{MO}, \cite{MMO}, among many others).

Recently, A.~Olevskii and the third author developed a new approach to Landau's theorems \cite{NO}. This approach provided a significantly simpler proof of Landau's results as well as several extensions of them. In particular, this technique does not require any study of operators or estimates of eigenvalues. The technique developed in \cite{NO} can be applied to various other settings, see e.g. \cite{LO}, \cite{G1}, \cite{G2}. In particular, it refines and extends results previously obtained by Landau's or Ramanathan and Steger's approach

One of the main goals in this paper is to develop and simplify further the approach introduced in \cite{NO} (see e.g., Lemma \ref{point} and its applications in this paper). We apply this approach to the study of Gabor systems, with the short time Fourier transform taking the role that the Fourier transform played in \cite{NO}. With this we obtain a simple proof of Theorem A. Moreover, this technique allows us to extend Theorem A in several new directions.

First, we find a more precise estimate of the distribution of $\L$ in cases where the generator $g$ satisfies some additional conditions. For example, if $g$ has some time frequency--concentration we obtain the following theorem as a particular case of a more general result.
\begin{theorem}\label{withest}
Fix $\alpha >0$. Let $\L\subset\R^2$ be a uniformly discrete sequence, with separation constant $\delta>0$, and let $g\in L^2(\R)$ be such that
\[
\int_{\R}|t|^{\alpha}|g(t)|^2dt<\infty\quad\textrm{and}\quad\int_{\R}|w|^{\alpha}|\hat{g}(w)|^2dw<\infty,
\]
where $\hat{g}$ denotes the Fourier transform of $g$.
\begin{itemize}
\item[1.] If $G(g,\L)$ is a Riesz sequence then
\[
\sup_{(a,b)\in\R^2}|\L\cap Q_{(a,b)}(R)|\leq  (2R+1)^2+C\rho_{\alpha}(R).
\]
\item[2.] If $G(g,\L)$ is a frame then
\[
\inf_{(a,b)\in\R^2}|\L\cap Q_{(a,b)}(R)|\geq  (2R-1)^2-C\rho_{\alpha}(R).
\]
\end{itemize}
Where $\rho_{\alpha}(R)$ is defined by
\[
\rho_{\alpha}(R)=\left\{
	\begin{array}{lll}
	R^{2-\alpha} &  0<\alpha< 2 \\
		\log R &  \alpha=2\\
1 &  \alpha>2
	\end{array}
\right.
\]
and $C$ denotes a positive constant which may depend on g, $\alpha$, $\delta$ and the Riesz sequence or frame bounds.
\end{theorem}

The quantity $\rho_{\alpha}(R)$ in Theorem \ref{withest} determines how much can the distribution of $\L$ deviate from the distribution of the integer lattice.
To see this, note that if $\L=\Z^2$ then
\[
(2R-1)^2<|\L\cap Q_{(a,b)}(R)|\leq (2R+1)^2,
\]
where the inequality on the left hand side is obtained when $R$ is an integer and the bound on the right hand side is approximated when $R$ is smaller then an integer, but arbitrarily close to it.

Next, we show that the Riesz sequence condition in Theorem A1 may be relaxed, so that the theorem holds for a larger class of 'under-complete' systems. (For the precise definition of uniformly minimal sequences see subsection 2.4).
\begin{theorem}\label{um thm}
If $G(g,\L)$ is a uniformly minimal sequence then (\ref{RS-RS}) holds. In general, the conclusion does not hold if the system merely is minimal.
\end{theorem}
In particular, when combined with the results from \cite{C2}, this implies that if $g\in L^1(\R)\cap L^2(\R)$ and $G(g,\L)$ is a Schauder basis then
$D(\L)=D^+(\L)=D^-(\L)=1$.

Further, we have the following theorem.
\begin{theorem}\label{lattice-thm}
Let $\L\subset \R^2$ be a lattice. Note that in this case $D(\L):=D^-(\L)=D^+(\L)=1/\textrm{det}(\L)$. We have,
\begin{itemize}
\item[1.] If $G(g,\L)$ is minimal then it is uniformly minimal and $D(\L)\leq 1$.
\item[2.] If $G(g,\L)$ is complete then $D(\L)\geq 1$.
\end{itemize}
\end{theorem}
We note that part 2 of Theorem \ref{lattice-thm} was obtained in \cite{RS}. We give a simple proof of this part as well.

Finally, Theorem A is known to hold when the single generator $g$ is replaced by a finite amount of generators, \cite{C3} (see also \cite{GRO}). We find that this is true for all the results formulated above.

\textbf{Remark 1.} All the results in this paper hold also in higher dimensions, with similar proofs (and an appropriate choice of exponents in the definition of $\rho_{\alpha}(R)$ in Theorem 1). We formulate and prove the results in dimension one for ease of notation alone.

The paper is organized as follows. In Section 2, we provide some background regarding systems of vectors in separable Hilbert spaces and the short time Fourier transform. In Section 3 we prove the key lemmas which will be used throughout the paper, as well as some auxiliary lemmas. In Section 4 we prove a more general version of Theorem A and obtain both it and Theorem \ref{withest} as corollaries. Theorem \ref{um thm} is proved in Section 5 and Theorem \ref{lattice-thm} in Section 6. Finally, in Section 7 we discuss the case of finitely many generators.

\section{Preliminaries}
\subsection{Notations}
For $(a, b)\in \R^2$, let $\mathcal{T}_a$ and $\mathcal{M}_b$ denote translation and modulation operators on $L^2(\R)$, that is
\[
(\mathcal{T}_ag)(t)= g(t- a) \quad\textrm{and}\quad (\mathcal{M}_bg)(t) = e^{2\pi ib t}g(t).
\]
Given a function $g\in L^2(\R)$, we denote $g_{ab}:=\mathcal{M}_b\mathcal{T}_ag$. So, for the Gabor system generated by $g\in L^2(\R)$ and the index set $\L\subset\R^2$ we have
\[
G(g,\L)=\{e^{2\pi i\mu t}g(t - \l)\}_{(\l,\mu)\in\L}=\{g_{\l\mu}\}_{(\l,\mu)\in\L}.
\]

For a function $g\in L^1(\R)$ let $\mathcal{F}g=\hat{g}$ denote the Fourier transform of $g$,
\[
\mathcal{F}g(w)=\hat{g}(w):=\int_{\R}g(t)e^{-2\pi i wt}dt,
\]
with the usual extension of $\mathcal{F}$ to a unitary operator on $L^2(\R)$.

As mentioned in the introduction, $Q_{(a,b)}(R)$ denotes the cube in $\R^2$ with center at $(a,b)$ and side length $2R$,
  \[
  Q_{(a,b)}(R)=\{(x,y)\in R^2: |x-a|,|y-b|\leq R\},
  \]
its complement $\R^2\setminus  Q_{(a,b)}(R)$ is denoted by $Q^c_{(a,b)}(R)$.

If $S\subset \R^2$ is a measurable set then $|S|$ denotes the Lebesgue measure of $S$. Given a discrete
set $\L\subset\R^2$, we write $|\L|$ for the number of points in $|\L|$. It will be clear from the context which option is being used.

Given a Hilbert space $H$ and a subset $E \subset H$, we let span$(E)$ denote the span of $E$, that is the
subspace of all finite linear combinations of elements from $E$.

\subsection{Density}

Let $\L\subset \R^2$. We say that $\L$ is uniformly discrete with separation constant $\delta > 0$ if
$\textrm{dist}(x,y)> \delta$ for all distinct $x,y \in \L$.
If $\L$ is uniformly discrete then the following limits exist \cite{Beu},
\[
D^+(\L):=\lim_{R\rightarrow\infty}\frac{\max_{(a,b)\in\R^2}|\L\cap Q_{(a,b)}(R)|}{(2R)^2};
\]
and
\[
D^-(\L):=\lim_{R\rightarrow\infty}\frac{\min_{(a,b)\in\R^2}|\L\cap Q_{(a,b)}(R)|}{(2R)^2}.
\]
$D^-(\L)$ and $D^+(\L)$ are called the lower and upper (Beurling) densities of $\L$, respectively. If $D^-(\L)=D^+(\L)$ we say that $\L$ has uniform density and denote it by $D(\L)$. These definitions go back to Beurling and Kahane. The choice of the cube in these definitions is arbitrary and is made for convenience; the values of these densities remain the same if the cube is replaced by any other convex set in $\R^2$. For a precise formulation and proof of this claim, see \cite{La}.

\subsection{The short time Fourier transform}

Let $\phi\in L^2(\R)$. The short-time Fourier transform (STFT) with window function $\phi$ is the operator acting on $L^2(\R)$ by
\[
{V_{\phi}g(x,y)} := \langle g, \phi_{xy} \rangle = \int_{\R}g(t)\overline{\phi(t-x)}e^{-2\pi i yt}dt.
\]
This transform is an important tool in the study of Gabor systems and time frequency analysis.
It is well known that if $\|\phi\|_{L^2(\R)}=1$ then $V_{\phi}$ is a unitary operator from $L^2(\R)$ into $L^2(\R^2)$. Moreover, we have
\begin{lemma}\label{basic}
Let $\phi,g\in  L^2(\R)$. Then:
\begin{itemize}
\item[1.] $V_{\phi}\mathcal{T}_{\lambda}g(x, y) = e^{-2\pi i \l y} V_{\phi}g(x-\l, y)$;
\item[2.] $V_{\phi}\mathcal{M}_{\mu}g(x, y)= V_{\phi}g(x, y-\mu)$.
\end{itemize}
In particular, the following equality holds,
\begin{equation}\label{translations}
|V_{\phi}g_{\l\mu}(x,y)|=|V_{\phi}g(x-\l,y-\mu)|.
\end{equation}
\end{lemma}
For more information on the STFT see e.g. \cite{GR}.

The following lemma was pointed out in \cite{RS} and is essentially an immediate consequence of the fact that the image of the Bergman-Fock transform is a space of entire functions.

\begin{lemma}\label{pointestimate}
Let $\phi(t)=e^{-|t|^2}$. For every $\delta>0$ there exists $C(\delta)$ such that for every $(\l,\mu), (x,y)\in \R^2$ and for every $G=V_{\phi}g$, $g\in L^2({\R})$, we have
\[
|G(x-\l,y-\mu)|^2\leq C(\delta)\int_{{Q_{(\l,\mu)}}(\delta)}|G(x-s,y-t)|^2dsdt.
\]
\end{lemma}

As an immediate consequence of Lemma \ref{pointestimate} we obtain the following.
\begin{lemma}\label{ballestimate}
Let $\phi(t)=e^{-|t|^2}$ and $0 < \delta $. There exists $C(\delta)$ depending only on $\delta$ such that the following holds: For every $\L\subset\R^2$ with separation constant $\delta$, every $G=V_{\phi}g$ where $g\in L^2({\R})$, and every cube $Q_{(a,b)}(R)$ we have

\begin{equation}\label{ballestimate1}
\sum_{\L\cap Q_{(a,b)}(R)}|G(x-\l,y-\mu)|^2\leq C(\delta)\int_{Q_{(a,b)}(R+\frac{1}{4})}|G(x-s,y-t)|^2dsdt,
\end{equation}
and
\begin{equation}\label{ballestimate2}
\sum_{\L\cap Q^c_{(a,b)}(R)}|G(x-\l,y-\mu)|^2\leq C(\delta)\int_{Q^c_{(a,b)}(R-\frac{1}{4})}|G(x-s,y-t)|^2dsdt.
\end{equation}
for all $(x,y)\in\R^2$.
\end{lemma}

The following lemma relates the time--frequency localization of a function to the decay of its Gabor transform.
We use the notation
\[
L^2_{\alpha}(\R^d)=\left\{F:\int_{\R^{{d}}}
\left(\sum_{l=1}^d|x_{l}|^{\alpha}\right)|F(x)|^2dx<\infty\right\}
\]
to denote the weighted $L^2$ space.

\begin{lemma}\label{inalpha}
Let $\phi(t)=e^{-|t|^2}$. If $g\in L^2(\R)$ satisfies
$
g,\hat{g}\in L^2_{\alpha}(\R),
$
then the Gabor transform of $g$ satisfies
\[
V_{\phi}g\in L^2_{\alpha}(\R^2)
\]
\end{lemma}
\begin{proof}
By Parseval's equality
\begin{align*}
\int_{\R^2}|x|^{\alpha}|V_{\phi}g(x,y)|^2dxdy&=
\int_{\R}|x|^{\alpha}\int_{\R}|{g}(t) \phi(t-x)|^2dtdx\\
&\leq C(\alpha)\int_{\R}(|t|^{\alpha}+|t-x|^{\alpha})\int_{\R}|{g}(t)\phi(t-x)|^2dtdx\\
&= C(\alpha)\big(\ {\|g\|}^2_{L^2(\R)}\|\phi\|^2_{L^2_{\alpha}(\R)}+\|\phi\|^2_{L^2(\R)}\|g\|^2_{L^2_{\alpha}(\R)}\big)<\infty.
\end{align*}

Next, we note that
\[
\int_{\R} {g}(t)\phi(t-x)e^{-2\pi i yt}dt=e^{-2\pi i yx}\int_{\R}\hat{g}(w)\phi(w+y)e^{-2\pi i xw}dw
\]
so we can use the same computation again to complete the proof.
\end{proof}

For additional results of this type see \cite{GR}
\subsection{Systems of vectors in a Hilbert space}
\subsubsection{Definitions}

Let $H$ be a separable Hilbert space. Throughout this paper we will be interested in several different types of systems in $H$.

\begin{definition}
A system of vectors $\{f_n\}\subset H$ is a \textit{Riesz
basis} (RB) in $H$ if there exists an orthonormal sequence $\{e_n\}\subset H$ and a bounded invertible operator
$T : H \rightarrow H $ such that $f_n = Te_n$ for all $n$.
\end{definition}

The following two definitions formalize the ideas of `over--complete basis' and `under--complete basis' respectively.

\begin{definition} A system of vectors $\{f_n\}\subset H$ is a \textit{frame} if
 \begin{equation}\label{frame}
 A\|f\|^2\leq \sum |\langle f,f_n\rangle|^2\leq B\|f\|^2\qquad\forall f\in H,
 \end{equation}
 where $A$ and $B$ are positive constants. $($The best possible constants $A$ and $B$ for which this inequality holds are called the lower and upper frame bounds$)$. A system which satisfies the right side inequality in (\ref{frame}) is called a Bessel sequence.
 \end{definition}

 \begin{definition}
 A system of vectors $\{f_n\}\subset H$ is a \textit{Riesz sequence} $($RS$)$ if it is a Riesz basis in the closure of its linear span. This is equivalent to the statement that,
 \begin{equation}\label{rb}
 A\sum |a_n|^2\leq {\left\|\sum a_nf_n \right\|}^2\leq B\sum |a_n|^2,\qquad\forall \{a_n\}\in l^2,
 \end{equation}
 where $A$ and $B$ are positive constants. $($The best possible constants $A$ and $B$ for which this inequality holds are called the lower and upper RS bounds$)$.

\end{definition}

We will use the following lemma, which follows from a simple duality argument.
\begin{lemma}\label{rb is rs and frame}
A Riesz basis is both a frame and a RS. In this case, the lower and upper RS bounds are equal to the lower and upper frame bounds.
\end{lemma}

The left inequality in the frame condition implies that every frame is complete. The opposite implication does not hold, in fact the frame condition is much stronger then mere completeness. Similarly, the left inequality in the RS condition implies that every RS is \textit{minimal}, i.e., every vector in the system lies outside the closed linear span of the rest of the vectors. The following definition provides for a somewhat stronger version of minimality.

\begin{definition} A system of vectors $\{f_n\}\subset H$ is
 \textit{uniformly minimal} if there exists some $\delta > 0$ such
that $\textrm{dist}(f_n, {\overline{\textrm{span}}}\{f_k\}_{k\neq n}) > \delta$ for all $n$.
\end{definition}

Every RS is, in particular, uniformly minimal. The opposite implication does not hold, i.e., uniform minimality is a much more relaxed notion than the RS condition.

\subsubsection{Dual systems and expansions}

It is well known that a system $\{f_n\}\subset H$ is minimal if and only if there exists a system $\{g_n\}\subset H$ such that the two systems are biorthogonal, i.e. $\langle f_n,g_m\rangle =\delta_{nm}$
 where $\delta_{nm}$ is the Kronecker
delta. The system $\{g_n\}$ is sometimes referred to as a dual system to $\{f_n\}$. A dual system is unique if and only if the system $\{f_n\}$ is complete. We will use the following.
 \begin{lemma}\label{um-dual}
A system $\{f_n\}\subset H$ is uniformly minimal if and only if there exists a \textit{bounded} system $\{g_n\}\subset H$ such that the two are biorthogonal.
 \end{lemma}

 If $\{f_n\}$ is a RB then the system $\{g_n\}$ above is also a RB which is called the \textit{dual Riesz basis} of $\{f_n\}$. In this case, every $f\in H$ can be decomposed into a series
 \[
 f=\sum\langle f,g_n\rangle f_n=\sum\langle f,f_n\rangle g_n.
 \]

In general, frames do not admit a dual which provides bi-orthogonality. Instead,
we use the notion of a dual frame. If $\{f_n\}$ is a frame in $H$, then there exists a frame $\{g_n\}\subset H$ such
that
every $f\in H$ can be decomposed into a series
 \[
 f=\sum\langle f,g_n\rangle f_n=\sum\langle f,f_n\rangle g_n.
 \]
The system $\{g_n\}$ is called a \textit{dual frame} of $\{f_n\}$.

For more information about systems of vectors in separable Hilbert spaces see e.g. \cite{young} or \cite{C4}.

\section{The main observations and some auxiliary lemmas}

Throughout the rest of the paper we fix $\phi=Ce^{-|t|^2}$ where the constant $C$ is chosen to satisfy $\|\phi\|_{L^2(\R)}=1$. This implies that the STFT with window function $\phi$ is unitary and that Lemmas \ref{ballestimate} and 4 hold.
\subsection{The main observations}
The following two lemmas are the key ingredients of the approach presented in this paper.
\begin{lemma}\label{point}
Let $W\subset L^2(\R)$ be a finite dimensional subspace and let $P_W$ be the orthogonal projection of $L^2(\R)$ on $W$. Then
\[
\int_{\R^2}\|P_W\phi_{xy}\|^2dxdy= \mathrm{dim}\,W.
\]
\end{lemma}
\begin{proof}
Let $M=\mathrm{dim}\,W$ and let $h_1,...,h_M$ be an orthonormal basis for $W$. Then, since we fixed $\phi$ above so that $\|\phi\|_{L^2(\R)}=1$ meaning the STFT is unitary, we have:
\[
\int_{\R^2}\|P_W\phi_{xy}\|^2dxdy=\int_{\R^2}\sum_m|\langle h_m,\phi_{xy}\rangle|^2dxdy
\]
\[
=\sum_m\int_{\R^2}|V_{\phi}h_m(x,y)|^2dxdy=\sum_m\|h_m\|^2=\mathrm{dim}\,W.
\]
\end{proof}

\begin{lemma}\label{anotherpoint}
Let $\{g_n\}$ and $\{h_n\}$ be dual Riesz bases of some closed subspace $W \subset L^2(\R)$. Denote $G_n(x, y) = V_{\phi}g_n(x, y) $
and $H_n(x, y) = V_{\phi}h_n(x, y) $.
Then,
$
\|P_W\phi_{xy}\|^2= \sum_n G_n(x, y)\overline{H_n(x, y)},
$
and therefore
\begin{equation*}
 0\leq \sum_n G_n(x, y)\overline{H_n(x, y)} \leq 1;\qquad\forall (x, y)\in \R^2.
 \end{equation*}
 \end{lemma}
 \begin{proof}
Let $P_W$ be the orthogonal projection onto $W$. We have,
\[
\sum_n G_n(x, y)\overline{H_n(x, y)}=\sum_n\langle g_n,\phi_{xy}\rangle_{L^2(\R)}\langle \phi_{xy},h_n\rangle_{L^2(\R)}
\]
\[
=\langle \big(\sum_n\langle \phi_{xy},h_n\rangle g_n\big), \phi_{xy}\rangle=\langle P_W\phi_{xy},\phi_{xy}\rangle=\|P_W\phi_{xy}\|^2.
\]
\end{proof}

\subsection{Auxiliary lemmas}
For $G\in L^2(\R^2)$ and $R>1$ denote
\[
I_G(R):=\int_{Q_0(R)}\int_{ Q^c_0(R+{\frac{1}{4}})}|G(t-x,s-y)|^2dtdsdxdy
\]
\begin{lemma}\label{err1}
Let $G\in L^2(\R^2)$.  then
\[
\frac{1}{R^2}I_G(R)\rightarrow 0 \qquad \textrm{as }R\rightarrow\infty.
\]
\end{lemma}
\begin{proof}
Let $A:= Q_0(R-\sqrt{R})$ and $A^c:=Q_0(R)\setminus A$. We note that $\int_{Q_0(R)}=\int_A+\int_{A^c}$ and write $I_G(R)=I_A+I_{A^c}$ accordingly.
We have,
\[
\frac{I_{A^c}}{{R^2}}
\leq\frac{|A^c|\|G\|^2}{R^2}\leq C\frac{R^{1.5}}{R^2}\longrightarrow 0,
\]
and
\begin{align*}
\frac{I_A}{{R^2}}  &\leq\frac{1}{R^2}\int_{ Q_0(R-\sqrt{R})}\int_{ Q^c_0(\sqrt{R})}|G(t,s)|^2dtdsdxdy\\
&\leq \int_{ Q^c_0(\sqrt{R})}|G(t,s)|^2dsdt\longrightarrow 0,
\end{align*}
since $G\in L^2(\R^2)$.
\end{proof}

Recall the notation,
\[
L^2_{\alpha}(\R^d)=\left\{F:\int_{\R^{{d}}} \left(\sum_{l=1}^d|x_{l}|^{\alpha}\right)|F(x)|^2dx<\infty\right\}.
\]

\begin{lemma}\label{err2}
If $G\in L^2_{\alpha}(\R)$ then
\[
I_G(R)\leq
\left\{
	\begin{array}{lll}
	C(\alpha)\|G\|^2_{L^2_{\alpha}(\R^2)}R^{2-\alpha}&  0<\alpha< 2 \\
		C(\alpha)\|G\|^2_{L^2_{\alpha}(\R^2)}\log R &  \alpha=2\\
C(\alpha)\|G\|^2_{L^2_{\alpha}(\R^2)}&  \alpha>2
	\end{array}
\right.
\]
where $C(\alpha)$ is a positive constant depending only on $\alpha$.
\end{lemma}

\begin{proof}
We have
\[
\int_{Q_0(R)}\int_{ Q^c_0(R+\frac{1}{4})}|G(t-x,s-y)|^2dsdtdxdy
\]
\[
\leq\int_{Q_0(R)}\int_{Q^c_0(R+\frac{1}{4})}\frac{|t-x|^{\alpha}+|s-y|^{\alpha}}{(R+\frac{1}{4}-|x|)^{\alpha}+(R+\frac{1}{4}-|y|)^{\alpha}}|G(t-x,s-y)|^2dsdtdxdy
\]
\[
\leq\|G\|^2_{L^2_{\alpha}(\R^2)}\int_{Q_0(R)}\frac{1}{(R+\frac{1}{4}-|x|)^{\alpha}+(R+\frac{1}{4}-|y|)^{\alpha}}dxdy
\]
\[
\leq C(\alpha)\|G\|^2_{L^2_{\alpha}(\R^2)}\int_0^{R}\int_0^{R}\frac{1}{(2R+\frac{1}{2}-x-y)^{\alpha}}dxdy.
\]
The result follows.
\end{proof}

   \section{Proofs for Theorem A and Theorem 1}

In this section we use Lemma \ref{point} to prove the following stronger version of Theorem A. Theorem A itself, as well as Theorem \ref{withest} follow from this theorem as corollaries.
\begin{theorem}\label{main}
Let $g\in L^2(\R)$ and let $\L\subset\R^2$ be a uniformly discrete sequence, with separation constant $\delta$. Denote $G:=V_{\phi}g$. Then,
\begin{itemize}
\item[1.] If $G(g,\L)$ is a Riesz sequence then
\[
\sup_{(a,b)\in\R^2}|\L\cap Q_{(a,b)}(R)|\leq  (2R+1)^2+CI_G\left(R+\frac{1}{4}\right).
\]
\item[2.] If $G(g,\L)$ is a frame then
\[
\inf_{(a,b)\in\R^2}|\L\cap Q_{(a,b)}(R)|\geq  (2R-1)^2-CI_G\left(R-\frac{1}{2}\right).
\]
\end{itemize}

Where $C$ denotes a positive constant which may depend on $\delta$ and the Riesz sequence or frame bounds.
\end{theorem}

\textbf{Remark 1. } Clearly, Theorem A follows from Theorem \ref{main} and Lemma \ref{err1} while Theorem \ref{withest} follows from Theorem \ref{main} combined with Lemmas \ref{inalpha} and \ref{err2}.

\textbf{Remark 2. } Keeping in mind the discussion after the formulation of Theorem \ref{withest}, we see that $I_G(R)$ determines how much the distribution of $\L$ can deviate from the distribution of the integer lattice.
$I_G(R)$ takes here the role that $\log R$ played in Landaus result for a finite union of intervals in \cite{La}; the role of the second term of convergence when $R$ tends to infinity.

We now turn to the proof of Theorem \ref{main}.

\subsection{Proof of Theorem \ref{main} part 1, Riesz sequences} Assume that $G(g,\L)$ is a Riesz sequence and denote its lower Riesz sequence bound by $A$.

Fix $(a,b)\in\R^2$ and $R>0$. Denote
\[
W=\textrm{span}\{g_{\l\mu}:(\l,\mu)\in\L\cap Q_{(a,b)}(R)\},
\]
and let $P=P_W$ be the orthogonal projection of $L^2(\R)$ onto $W$.\\

$\:$

\textbf{Step I.} Clearly, the (finite) system $\{g_{\l,\mu}:(\l,\mu)\in \L\cap Q_{(a,b)}(R)\}$ is a Riesz basis for $W$ with lower Riesz basis bound $A$. In particular, this implies that
\begin{equation}\label{dim}
\mathrm{dim}\,W=|\L\cap Q_{(a,b)}(R)|
\end{equation}
Moreover, it follows from Lemma \ref{rb is rs and frame} that $\{g_{\l,\mu}:(\l,\mu)\in \L\cap Q_{(a,b)}(R)\}$ is a frame in $W$ with lower frame bound $A$. So, for every $(x,y)\in\R^2$,
\begin{align*}
\|P\phi_{xy}\|^2&\leq\frac{1}{A}\sum_{ \L\cap Q_{(a,b)}(R)}|\langle P\phi_{xy},g_{\l\mu}\rangle|^2
=\frac{1}{A}\sum_{ \L\cap Q_{(a,b)}(R)}|\langle \phi_{xy},g_{\l\mu}\rangle|^2\\&=\frac{1}{A}\sum_{ \L\cap Q_{(a,b)}(R)}|G(x-\l,y-\mu)|^2,
\end{align*}
where the last equality follows from (\ref{translations}). Since $\L$ is uniformly discrete, say with separation constant $\delta>0$, we may apply (\ref{ballestimate1}). We find that
\begin{equation}\label{almostI}
\|P\phi_{xy}\|^2\leq\frac{C(\delta)}{A}\int_{ Q_{(a,b)}(R+\frac{1}{4})}|G(x-s,y-t)|^2dsdt.
\end{equation}

$\:$

\textbf{Step II.} For every $(x,y)\in\R^2$ we have $$\|P\phi_{xy}\|^2\leq 1.$$ Integrating over $Q_{(a,b)}(R+\frac{1}{2})$ we obtain,
\[
\int_{Q_{(a,b)}(R+\frac{1}{2})}\|P\phi_{xy}\|^2dxdy\leq  (2R+1)^2
\]
So,
\[
\int_{\R^2}\|P\phi_{xy}\|^2dxdy\leq (2R+1)^2+\int_{Q^c_{(a,b)}(R+\frac{1}{2})}\|P\phi_{xy}\|^2dxdy.
\]
By Lemma \ref{point} the left hand side is equal to $\mathrm{dim}\,W$ which, by (\ref{dim}), is equal to $|\L\cap Q_{(a,b)}(R)|$.
Substituting (\ref{almostI}) in the right hand side, we get
\[
|\L\cap Q_{(a,b)}(R)|\leq (2R+1)^2+\frac{C(\delta)}{A}\int_{Q^c_{(a,b)}(R+\frac{1}{2})}\int_{ Q_{(a,b)}(R+\frac{1}{4})}|G(x-t,y-s)|^2dtdsdxdy.
\]
The result for Riesz sequences now follows from a simple change of variables in the integral.

\subsection{Proof of Theorem \ref{main} part 2, frames} Assume that $G(g,\L)$ is a frame, denote its dual frame by $\{\widetilde{g}_{\l\mu}\}$.

Fix $(a,b)\in\R^2$ and $R>0$. Denote
\[
W=\textrm{span}\{\widetilde{g}_{\l\mu}:{(\l,\mu)}\in\L\cap Q_{(a,b)}(R)\},
\]
and let $P$ and $P^+$ be the orthogonal projections of $L^2(\R)$ on $W$ and its orthogonal complement respectively.

Clearly,
\begin{equation}\label{dimframe}
\mathrm{dim}\,W\leq |\L\cap Q_{(a,b)}(R)|
\end{equation}

$\:$

\textbf{Step I.} This step follows an idea from \cite{GRO}. Since $\{\widetilde{g}_{\l\mu}\}$ is a frame, say with upper bound $B$, we have for every $(x,y)\in\R^2$,
\begin{align*}
\|P^+\phi_{xy}\|^2\leq& \|\phi_{xy}-\sum_{\L\cap Q_{(a,b)}(R)}\langle \phi_{xy}, g_{\l\mu}\rangle\widetilde{g}_{\l\mu}\|^2
=\|\sum_{\L\cap Q^c_{(a,b)}(R)}\langle \phi_{xy}, g_{\l\mu}\rangle\widetilde{g}_{\l\mu}\|^2\\
\leq&B\sum_{\L\cap Q^c_{(a,b)}(R)}|\langle \phi_{xy}, g_{\l\mu}\rangle|^2=B\sum_{\L\cap Q^c_{(a,b)}(R)}| G(x-\l,y-\mu)|^2.
\end{align*}

Since $\L$ is uniformly discrete, say with separation constant $\delta$, we may apply (\ref{ballestimate2}). We find that
\begin{equation}\label{almostIf2}
\|P^+\phi_{xy}\|^2\leq BC(\delta)\int_{ Q^c_{(a,b)}(R-\frac{1}{4})}|G(x-t,y-s)|^2dtds.
\end{equation}

$\:$

\textbf{Step II.} For every $(x,y)\in\R^2$ we have $1=\|P\phi_{xy}\|^2+\|P^+\phi_{xy}\|^2$. Integrating over $Q_{(a,b)}(R-\frac{1}{2})$ we obtain,
\[
(2R-1)^2=\int_{Q_{(a,b)}(R-\frac{1}{2})}\|P\phi_{xy}\|^2dxdy+\int_{Q_{(a,b)}(R-\frac{1}{2})}\|P^+\phi_{xy}\|^2dxdy.
\]
We first note that, by lemma \ref{point} and (\ref{dimframe})
\[
\int_{Q_{(a,b)}(R-\frac{1}{2})}\|P\phi_{xy}\|^2dxdy\leq\textrm{dim}W\leq |\L\cap Q_{(a,b)}(R)|.
\]
Next, (\ref{almostIf2}) implies that
\[
\int_{Q_{(a,b)}(R-\frac{1}{2})}\|P^+\phi_{xy}\|^2dxdy\leq BC(\delta)\int_{ Q_{(a,b)}(R-\frac{1}{2})}\int_{Q^c_{(a,b)}(R-\frac{1}{4})}|G(x-t,y-s)|^2dtdsdxdy.
\]
Combining these estimates we obtain,
\[
|\L\cap Q_{(a,b)}(R)|\geq (2R-1)^2-BC(\delta)\int_{ Q_{(a,b)}(R-\frac{1}{2})}\int_{Q^c_{(a,b)}(R-\frac{1}{4})}|G(x-t,y-s)|^2dtdsdxdy.
\]
The result for frames now follows from a simple change of variables in the integral. \[\qquad\qquad\qquad\qquad\qquad\qquad\qquad\qquad\qquad\qquad\qquad\qquad\qquad\qquad\qquad\qquad\qquad\qquad\qed\]

\section{Uniform minimality, a proof for Theorem \ref{um thm}}

In this section we use Lemma \ref{anotherpoint} to prove Theorem \ref{um thm}. This proof follows the corresponding proof from \cite{NO} and presents a different aspect of the approach in this paper.

Fix $\epsilon > 0$ and choose $b > 0$ such that
\begin{equation}\label{eps}
\int_{Q^c_0(b)}|V_{\phi}g(x,y)|^2dxdy<\epsilon^2.
\end{equation}

Since $G(g,\L)$ is uniformly minimal, by Lemma \ref{um-dual} there exists a bounded system $\{h^{\l\mu}\}$ such that the two systems are biorthogonal. Denote the bound of this system by $B$. (Note that in general $\{h^{\l\mu}\}$ is not a Gabor system but rather is just a sequence of functions indexed by $\L$).

For $(a,b)\in \R^2$ and $R>0$, consider the space
\[
 W = \textrm{span}\{g_{\l\mu} : (\l,\mu) \in\L\cap Q_{(a,b)}(R)\}
 \]and denote by
$P=P_W$ the orthogonal projection from $L^2(\R)$ onto this space. The (finite) systems $\{g_{\l\mu}\}_{\L\cap Q_{(a,b)}(R)}$ and $\{Ph^{\l\mu}\}_{\L\cap Q_{(a,b)}(R)}$ are biorthogonal Riesz bases in $W$. Denote $G_{\l\mu}=V_{\phi}g_{\l\mu}$ and $H_{\l\mu}=V_{\phi}(Ph^{\l\mu})$. By Lemma \ref{anotherpoint} we have,
\[
 0\leq \sum_{\L\cap Q_{(a,b)}(R)} G_{\l\mu}(x, y)\overline{H_{\l\mu}(x, y)} \leq 1,\qquad\forall (x, y)\in \R^2.
\]
Integrating over $Q_{(a,b)}(R+b)$, we obtain
\begin{equation}\label{umest}
\sum_{\L\cap Q_{(a,b)}(R)}\int_{Q_{(a,b)}(R+b)} G_{\l\mu}(x, y)\overline{H_{\l\mu}(x, y)}dxdy\leq (2(R+b))^2.
\end{equation}
Since the STFT is unitary and the systems are biorthogonal,
\begin{equation}\label{umest2}
\int_{\R^2} G_{\l\mu}(x, y)\overline{H_{\l\mu}(x, y)}dxdy=\langle g_{\l\mu},Ph^{\l\mu}\rangle_{L^2(\R)}=\langle g_{\l\mu},h^{\l\mu}\rangle_{L^2(\R)}=1.
\end{equation}
Substituting in (\ref{umest}) and using the triangle inequality we have
\[
|\L\cap Q_{(a,b)}(R)|-\sum_{\L\cap Q_{(a,b)}(R)}\big|\int_{Q^c_{(a,b)}(R+b)} G_{\l\mu}(x, y)\overline{H_{\l\mu}(x, y)}dxdy\big|
\]
\begin{equation}\label{thisone}
\leq (2(R+b))^2.
\end{equation}
Since the dual system is bounded by $B$, $$\|H_{\l\mu}\|^2_{L^2(\R^2)}=\|Ph^{\l\mu}\|^2_{L^2(\R)}\leq B^2.$$ So, Cauchy inequality and Lemma \ref{basic} imply that each summand on the left hand side satisfies,
\begin{align*}
\big|\int_{ Q^c_{(a,b)}(R+b)} G_{\l\mu}(x, y)\overline{H_{\l\mu}(x, y)}dxdy\big|^2&\leq
B^2\int_{Q^c_{(a,b)}(R+b)}| G(x-\l, y-\mu)|^2dxdy\\
&\leq B^2\int_{Q^c_{0}(b)}| G(x, y)|^2dxdy,
\end{align*}
where the last inequality follows from the fact that $(\l,\mu)\in\L\cap Q_{(a,b)}(R)$. Now, (\ref{eps}) implies that the last expression is less then $B^2\epsilon^2$. Substituting in (\ref{thisone}), we get
\[
(1-B\epsilon)|\L\cap Q_{(a,b)}(R)|\leq (2(R+b))^2.
\]
Dividing both sides by $(2R)^2$ and letting first $R$ tend to infinity and then $\epsilon$ tend to zero, we obtain the required result.

Finally, in the following proposition we use a result from \cite{OU} to show that in general, the 'uniformity' condition is necessary for Theorem 2 to hold.

\begin{proposition}
Fix $d>0$. There exist $g\in L^2(\R)$ and $\L\subset\R^2$ with $D^+(\L)>d$ such that $G(g,\L)$ is minimal.
\end{proposition}
\begin{proof}
Let $a > 0$ be arbitrarily small. It is shown in \cite{OU} that there exists $\L\subset\R$ such that
$\{e^{2\pi i \l t}\}_{\l\in\L}$ is minimal in $L^2(-a,a)$ and $D^+(\L) >1$ (where $D^+(\L)$ is the upper density of $\L$ as a subset of the line). Let $K =(2a\Z)\times \L$ and $g = 1\!\!1_{[-a,a]}$ be the indicator of the interval ${[-a,a]}$. It follows that
$G(g,K)$ is minimal in $L^2(\R)$ and that $D^+(K) \geq 1/(2a)$, which can be made arbitrarily
large by choosing $a$ small enough.
\end{proof}

 \section{Lattice index set, a proof for Theorem \ref{lattice-thm}}

 In this section we prove Theorem \ref{lattice-thm}. To this end, recall the notations $\mathcal{M}_a$ and $\mathcal{T}_b$ defined in Subsection 2.1. Note that, given $a,b\in \R$, there exists $\xi=\xi(a,b)\in\C$, with $|\xi|=1$, such that
 \[
 \mathcal{M}_a\mathcal{T}_bg=\xi\mathcal{T}_b\mathcal{M}_ag\qquad \forall g\in L^2(\R).
 \]
 \subsection{A proof for Theorem \ref{lattice-thm} part 1}
 Let $v=(v_1,v_2),w=(w_1,w_2)\in\R^2$ be such that $\L=\{nv+kw:(n,k)\in\Z^2\}$ Denote the elements of $G(g,\L)$ by
 \[
 g^{nk}:=\mathcal{M}_{nv_2+kw_2}\mathcal{T}_{nv_1+kw_1}g,\qquad (n,k)\in\Z^2.
 \]
 Since the system $G(g,\L)$ is minimal it has a dual system; in particular, there exists a function $h\in L^2(\R)$ such that
 $ \langle g,h \rangle =1$
and $ \langle g^{nk},h \rangle = 0$ for all $(n,k)\neq (0,0)$ in $\Z^2$.

Define
\[
h^{nk}:=\mathcal{M}_{nv_2+kw_2}\mathcal{T}_{nv_1+kw_1}h,\qquad (n,k)\in\Z^2.
\]
Note that  $ \langle g^{nk},h^{nk}\rangle =\langle g,h \rangle =1$, while for $(m,l)\neq (n,k)$ there exists $\xi:=\xi(m,l,n,k)$ such that
\begin{align*}
\langle g^{ml},h^{nk}\rangle & =
\xi\langle \mathcal{M}_{(m-n)v_2+(l-k)w_2}\mathcal{T}_{(m-n)v_1+(l-k)w_1}g,h\rangle=0.
\end{align*}
 It follows that $\{ g^{nk}\}$ and $\{ h^{nk}\}$ are biorthogonal sequences. Since for all $(n,k)\in\Z^2$ we have $\|h^{nk}\|=\|h\|$, the system $\{ h^{nk}\}$ is bounded, which by Lemma \ref{um-dual} implies that $G(g,\L)$ is uniformly minimal. Theorem \ref{um thm} implies now that $D^+(\L)\leq 1$.
 \subsection{A proof for Theorem \ref{lattice-thm} part 2}
Let $v,w\in\R^2$ be such that $\L=\{nv+kw:(n,k)\in\Z^2\}$ and denote $A=\{\xi v+\eta w:(\xi,\eta)\in[0,1]^2\}$.

$\:$

\textbf{Step I.} Using the terminology of \cite{RS}, in this step we prove that $G(g,\L)$ has the \textit{homogeneous approximation property}, in much the same way as is done in \cite{RS}.

 Fix $\epsilon>0$. Since $G(g,\L)$ is complete and $A$ is compact, there exists $d'>0$ such that for all $(x,y)\in A$ the function $\phi_{xy}$ can be approximated within $\sqrt{\epsilon}$ by a function from
 \[
 \textrm{span}\{g_{\l\mu}:(\l,\mu)\in\L\cap Q_0(d')\}.
 \]
Since $\L$ is a lattice, it follows that there exists $d>0$ such that for all $(x,y)\in \R^2$ the function $\phi_{xy}$ can be approximated up to $\sqrt{\epsilon}$ by a function from
\[
 \textrm{span}\{g_{\l\mu}:(\l,\mu)\in\L\cap Q_{(x,y)}(d)\}.
\]

$\:$

\textbf{Step II.} In the terminology of \cite{RS}, in this step we prove that if $G(g,\L)$ has the \textit{homogeneous approximation property} then $D^-(\L)\geq1$.

Let $d$ be the number chosen in Step I and fix $(a,b)\in\R^2$ and $R>d$. Denote
\[
W=\textrm{span}\{{g}_{\l\mu}:{(\l,\mu)}\in\L\cap Q_{(a,b)}(R)\},
\]
and let $P$ be the orthogonal projections of $L^2(\R)$ onto W.
Clearly,
\begin{equation}\label{dimcomp}
\mathrm{dim}\,W\leq |\L\cap Q_{(a,b)}(R)|.
\end{equation}
It follows from our choice of $d$ that for all $(x,y)\in Q_{(a,b)}(R-d)$ we have
\begin{equation}\label{projcomp}
\|P\phi_{xy}\|^2\geq 1-\epsilon.
\end{equation}
 Integrating over $Q_{(a,b)}(R-d)$, we obtain
\[
(1-\epsilon)(2R-2d)^2\leq\int_{Q_{(a,b)}(R-d)}\|P\phi_{xy}\|^2dxdy.
\]
By Lemma \ref{point} and (\ref{dimcomp})
\[
\int_{Q_{(a,b)}(R-d)}\|P\phi_{xy}\|^2dxdy\leq\mathrm{dim}\,W\leq |\L\cap Q_{(a,b)}(R)|.
\]
Combining these we obtain
\[
(1-\epsilon)(2R-2d)^2\leq |\L\cap Q_{(a,b)}(R)|.
\]
Dividing by $(2R)^2$ and letting first $R$ tend to infinity and then $\epsilon$ tend to zero, the result follows.

\section{Finite sets of generators}

In this section we formulate a version of our results in the case that the system considered is generated by a finite number of generators. To avoid repetition we omit the proofs of these theorems as they are much the same as in the case of a single generator.

\subsection{Definitions}
We will use the following definitions. Let $\{\L_n\}_{n=1}^N$ be a finite family of sequences $\L_n\subset \R^2$. We say that $\{\L_n\}$ is a uniformly discrete family if each $\L_n$ is uniformly discrete.

For a uniformly discrete family $\{\L_n\}$ the following limits exist,
\[
D^+(\{\L_n\}):=\lim_{R\rightarrow\infty}\frac{\max_{(a,b)\in\R^2}\sum_{n=1}^N|\L_n\cap Q_{(a,b)}(R)|}{(2R)^2};
\]
and
\[
D^-(\{\L_n\}):=\lim_{R\rightarrow\infty}\frac{\min_{(a,b)\in\R^2}\sum_{n=1}^N|\L_{{n}}\cap Q_{(a,b)}(R)|}{(2R)^2}.
\]

We call $D^-(\{\L_n\})$ and $D^+(\{\L_n\})$ the lower and upper densities of the family $\{\L_n\}$ respectively.

\subsection{Results}
We have the following extension of Theorem 4.
\begin{theorem4'}\label{main'}
Let $g_1,...,g_N$ be distinct functions in $L^2(\R)$ and let $\{\L_n\}_{n=1}^N$ be a uniformly discrete family of real sequences. Denote $G_n:=V_{\phi}g_n$, $1\leq n\leq N$. There exists $C>0$ such that,
\begin{itemize}
\item[1.] If $\bigcup_{n=1}^N G(g_n,\L_n)$ is a Riesz sequence then
\[
\sup_{(a,b)\in\R^2}\sum_{n=1}^N|\L_n\cap Q_{(a,b)}(R)|\leq  (2R+1)^2+C\sum_{n=1}^NI_{G_n}(R+\frac{1}{4}).
\]
\item[2.] If $\bigcup_{n=1}^N G(g_n,\L_n)$ is a frame then
\[
\inf_{(a,b)\in\R^2}\sum_{n=1}^N|\L_n\cap Q_{(a,b)}(R)|\geq  (2R-1)^2-C\sum_{n=1}^NI_{G_n}(R-\frac{1}{2}).
\]
\end{itemize}
\end{theorem4'}

A similar extensions of Theorems A follows.

\begin{theorema'}
Let $g_1,...,g_N$ be distinct functions in $L^2(\R)$ and let $\{\L_n\}_{n=1}^N$ be a uniformly discrete family of real sequences.
\begin{itemize}
\item[1.] If $\bigcup_{n=1}^N G(g_n,\L_n)$ is a Riesz sequence then
$
D^+(\{\L_n\})\leq 1,
$
\item[2.] If $\bigcup_{n=1}^N G(g_n,\L_n)$ is a frame then
$
D^-(\{\L_n\})\geq 1.
$
\end{itemize}
\end{theorema'}

\textbf{Remark 2. }Suppose that in part 1 of Theorem A' we had considered the quantity
$\sum_{n=1}^N D^+(\L_n)$ instead of the upper
density $D^+(\{\L_n\})$. Since the subset of a RS is itself a RS, Theorem A easily implies that
\[
D^+(\{\L_n\})\leq \sum_{n=1}^N D^+(\L_n)\leq N.
\]
However, the estimate on the right hand side is the best that can be made. As an example of this fact consider the sets
\[
\L^*_n:=\left\{(x,y)\in \Z^2:\textrm{Arg}(x,y)\in \left(\frac{2\pi (n-1)}{N},\frac{2\pi n}{N}\right]\right\}.\qquad n=1,...,N
\]
Then the distinct functions $g_n:=e^{2\pi i nt}1\!\!1_{[0,1]}$ and the family $\L_n:=\L^*_n-(0,n)$.

$\:$

In addition, we have the following extensions of Theorems 2 and 3.
\begin{theorem2'}\label{um thm'}
If $\bigcup_{n=1}^N G(g_n,\L_n)$ is uniformly minimal then $
D^+(\{\L_n\})\leq 1
$. In general, the conclusion does not hold if the system merely is minimal.
\end{theorem2'}

\begin{theorem3'}\label{lattice-thm'}
Let $g_1,...,g_N\in L^2(\R)$, $\L\subset \R^2$ be a lattice and $x_1,...,x_N$ be distinct points in $\R^2$. Denote $\L_n:=x_n+\L$. Then,
\begin{itemize}
\item[1.] If $\bigcup_{n=1}^N G(g_n,\L_n)$ is minimal then it is uniformly minimal and $D(\L)\leq 1/N$.
\item[2.] If $\bigcup_{n=1}^N G(g_n,\L_n)$ is complete then $D(\L)\geq 1/N$.
\end{itemize}
\end{theorem3'}

\section{Acknowledgments}
The work on this project was done during an REU program at Kent State University in the summer of 2014. We wish to thank Kent State University and in particular the organizer Jenya Soprunova for an excellent REU program and for providing us with the opportunity to meet and work on this project.

\bibliographystyle{amsplain}

\begin{thebibliography}{10}

\bibitem{Beu} A.~Beurling, \textit{The Collected Works of Arne Beurling}, volume vol.2 of Contemporary Mathematics.
Birkhauser, 1989.
\bibitem{C3} O.~Christensen, B.~Deng, and C.~Heil,
\textit{Density of Gabor frames}
, Applied and Computational Harmonic Analysis,
7(3)
, 292--304., (1999).
\bibitem{C2} B.~Deng and C.~Heil,
\textit{Density of Gabor Schauder bases}, Wavelet Applications in Signal
and Image Processing VIII (San Diego, CA, 2000), Proc. SPIE
4119, Bellingham, WA, 153--164, (2000).
\bibitem{GR} K.~Gr\"{o}chenig, \textit{Foundations of Time-Frequency Analysis}, Springer, 2013.
\bibitem{GRO} K.~Gr\"{o}chenig and H.~Razafinjatovo, \textit{On Landau's necessary density conditions
for sampling and interpolation of band-limited function},  Journal of the London Mathematical Society, 54(3), 557--565, (1996).
\bibitem{G1} K.~Gr\"{o}chenig,  J.~L.~Romero, J.~Unnikrishnan, and M.~Vetterli, \textit{On minimal trajectories for mobile sampling of bandlimited fields} Applied and Computational Harmonic Analysis, 39(3), 487--510, (2015).
\bibitem{G2} K.~Gr\"{o}chenig and A.~Klotz, \textit{What is Variable Bandwidth}, arXiv preprint, arXiv:1512.06663 (2015).
\bibitem{C1} C.~Heil, \textit{History and evolution of the density theorem for Gabor frames}, Journal of Fourier Analysis and Applications 13(2),  113--166, 2007.
\bibitem{C4} C.~Heil, \textit{A Basis Theory Primer: Expanded Edition} Springer, (2010).
\bibitem{La} H.~Landau, \textit{Necessary density conditions for sampling and interpolation of certain entire functions}, Acta
Mathematica, 117(1), 37--52, 1967.
\bibitem{LO} N.~Lev and J.~Ortega-Cerda \emph{Equidistribution estimates for Fekete points on complex manifolds}, arXiv preprint, arXiv:1212.4679.
\bibitem{MMO} N.~Marco, X.~Massaneda and J.~Ortega-Cerda, \textit{Interpolating and sampling sequences for entire functions}, Geometric and Functional Analysis GAFA, 13(4), 862--914, (2003).
\bibitem{MO} J.~Marzo and J.~Ortega-Cerda, \textit{Equidistribution of Fekete points on the sphere}, Constructive Approximation, 32(3), 513--521, (2010).
\bibitem{NO} S.~Nitzan and A.~Olevskii, \textit{Revisiting landaus density theorems for Paley-Wiener spaces} Comptes Rendus
Mathematique, 350(910), 509--512, 2012.
\bibitem{OU} A.~Olevskii and A.~Ulanovskii, \textit{Interpolation in Bernstein and Paley--Wiener spaces}, Journal of Functional
Analysis, 256(10), 3257--3278, 2009.
\bibitem{RS} J.~Ramanathan and T.~Steger, \textit{Incompleteness of sparse coherent states}, Applied and Computational
Harmonic Analysis, 2(2), 148--153, 1995.
\bibitem{young} R.~M.~Young, \textit{An Introduction to Non-Harmonic Fourier Series, Revised Edition 93}, Academic Press, 2001.
\end{thebibliography}
\def\cprime{$'$} \def\cprime{$'$} \def\cprime{$'$}
\providecommand{\bysame}{\leavevmode\hbox to3em{\hrulefill}\thinspace}
\providecommand{\MR}{\relax\ifhmode\unskip\space\fi MR }
\providecommand{\MRhref}[2]{%
  \href{http://www.ams.org/mathscinet-getitem?mr=#1}{#2}
}
\providecommand{\href}[2]{#2}

  \end{document}